\documentclass[12pt,reqno]{amsart}
\usepackage[margin=1.5in]{geometry}
\usepackage{graphicx, amsmath, amssymb, amscd, amsthm, euscript, amsfonts,color, bm, esint, mathtools}
\usepackage[utf8]{inputenc}
\usepackage[english]{babel}
\usepackage{verbatim}
\usepackage[colorlinks=true,urlcolor=blue,linkcolor=blue,citecolor=blue,psdextra]{hyperref}

\usepackage{bookmark}
\usepackage{psfrag}

\newtheorem{theorem}{Theorem}[section]
\newtheorem{lemma}[theorem]{Lemma}
\newtheorem{proposition}[theorem]{Proposition}
\newtheorem{corollary}[theorem]{Corollary}

\theoremstyle{remark}
\newtheorem{remark}[theorem]{Remark}

\theoremstyle{definition}
\newtheorem{definition}[theorem]{Definition}

\numberwithin{equation}{section}

\DeclareMathOperator{\Ad}{{\rm Ad}}          

\DeclareMathOperator{\Lie}{{\rm Lie}}        
\DeclareMathOperator{\sign}{sign}

\newcommand{\field}[1]{\mathbb{#1}}
\newcommand{\R}{\field{R}}
\newcommand{\N}{\field{N}}
\newcommand{\Q}{\field{Q}}
\newcommand{\Z}{\field{Z}}

\newcommand{\DTA}[1]{\mathrm{DT}(#1)}
\newcommand{\DTB}[1]{\mathrm{DT}^\prime(#1)}

\newcommand{\mbz}{\boldsymbol{z}}

\def\cH{\mathcal{H}}
\def\cL{\mathcal{L}}
\def\cN{\mathcal{N}}
\def\cE{\mathcal{E}}
\def\cM{\mathcal{M}}
\def\cT{\mathcal{T}}

\def\te{{\tilde e}}

\newcommand{\lief}{\mathfrak{f}}

\providecommand{\abs}[1]{\lvert#1\rvert}
\providecommand{\Abs}[1]{\Bigl\lvert #1 \Bigr\rvert}
\providecommand{\norm}[1]{\lVert#1\rVert}

\DeclareMathOperator{\SL}{SL}

\DeclareMathOperator{\M}{M}
\DeclareMathOperator{\SO}{SO}
\newcommand{\SOd}{\SO(\cT)}
\DeclareMathOperator{\diag}{diag}
\newcommand{\inv}{^{-1}}
\newcommand{\cl}[1]{\overline{#1}}

\DeclareMathOperator{\vol}{vol}
\DeclareMathOperator{\Span}{span}

\usepackage[utf8]{inputenc}

\numberwithin{equation}{section}

\title[Expanding translates of shrinking submanifolds]{Limit distributions of expanding translates of shrinking submanifolds and non-improvability of Dirichlet's approximation theorem}
\author{Nimish A. Shah}

\address{The Ohio State University, Columbus, OH 43210; email: shah@math.osu.edu}
\author{Pengyu Yang}
\address{Morningside Center of Mathematics, Chinese Academy of Sciences, Beijing 100190; email:yangpengyu@amss.ac.cn}
\thanks{N.A. Shah was partially supported by NSF grant DMS-1700394.}
\thanks{P. Yang is supported by National Key R\&D Program of China 2022YFA1007500 and NSFC grant 22AAA00245.}

\begin{document}

\begin{abstract}
    On the space $\cL_{n+1}$ of unimodular lattices in $\R^{n+1}$, we consider the standard action of $a(t)=\diag(t^n,t^{-1},\ldots,t^{-1})\in \SL(n+1,\R)$ for $t>1$. Let 
    $M$ be a nondegenerate submanifold of an expanding horospherical leaf in $\cL_{n+1}$. 
    We prove that for all $x\in M\setminus E$ and $t>1$, if $\mu_{x,t}$ denotes the normalized Lebesgue measure on the ball of radius $t^{-1}$ around $x$ in $M$, then the translated measure $a(t)\mu_{x,t}$ gets equidistributed in $\cL_{n+1}$ as $t\to\infty$, where $E$ is a union of countably many lower dimensional submanifolds of $M$. In particular, if $\mu$ is an absolutely continuous probability measure on $M$, then $a(t)\mu$ gets equidistributed in $\cL_{n+1}$ as $t\to\infty$. This result implies the non-improvability of Dirichlet's Diophantine approximation theorem for almost every point on a $C^{n+1}$-submanifold of $\R^n$ satisfying a non-degeneracy condition, answering a question arising from the work of Davenport and Schmidt (1969).
\end{abstract}

\maketitle

\section{Introduction}

\label{sec:intro}

After Davenport and Schmidt~\cite{Davenport+Schmidt:Dirichlet}, given $0<\lambda\leq 1$, we say that $\mbz=(z_1,\ldots,z_n)\in\R^n$ is $\DTA{\lambda}$ if for each sufficiently large $N\in\N$, there exist integers $q_1,\ldots,q_n$ and $p$ such that
\begin{equation} \label{eq:DTA}
    \abs{(q_1z_1+\ldots+q_nz_n)-p}\leq \lambda/N^{n} \text{ and } 0<\max_{1\leq i\leq n} \abs{q_i}\leq \lambda N.
\end{equation}
In a dual manner, we say that $\mbz\in\R^n$ is $\DTB{\lambda}$ if for each sufficiently large $N\in\N$ there exist integers $q$ and $p_1,\ldots,p_n$ such that
\begin{equation} \label{eq:DTB}
 \max_{1\leq i\leq n}  \abs{qz_i-p_i}\leq \lambda/N \text{ and } 0<\abs{q}\leq \lambda N^n.
\end{equation}

Dirichlet's simultaneous approximation theorem states that every $\mbz\in\R^n$ is $\DTA{1}$ and $\DTB{1}$.
Davenport and Schmidt~\cite{Davenport+Schmidt:Dirichlet} proved that for any $\lambda<1$, almost every $\mbz\in\R^n$ is not $\DTA{\lambda}$ and not $\DTB{\lambda}$.
In other words, Dirichlet's theorem cannot be improved for almost all $\mbz\in\R^n$.
In~\cite{DS:curve}, they showed that for almost every $z\in \R$, the vector $\mbz=(z,z^2)\in\R^2$ is not $\DTA{1/4}$, opening an investigation of whether almost all points on a sufficiently curved submanifold in $\R^n$ are not $\DTA{\lambda}$ for any $\lambda<1$. The question was taken up in \cite{Baker:curves,Dodson:manifolds,Bugeaud:poly}, where several non-improvability results were obtained for small $\lambda>0$.
Later Kleinbock and Weiss~\cite{Kleinbock+Weiss:Dirichlet} 
reformulated this question in terms of dynamics on homogeneous spaces using an observation due to Dani~\cite{Dani:divergent} relating simultaneous Diophantine approximation to asymptotic properties of individual orbits of diagonal subgroups. Using the non-divergence techniques from \cite{Klein+Mar:Annals98}, they~\cite{Kleinbock+Weiss:Dirichlet} proved that for any `$l$-nondegenerate' differentiable map $\psi$ from an open set $\Omega\subset\R^d$ to $\R^n$, there exists $\lambda>0$ such that $\psi(s)$ is not $\DTA{\lambda}$ for Lebesgue almost every $s\in\Omega$.

In~\cite{Shah:Dirichlet} by proving an equidistribution result for expanding translates of analytic curve segments on the space of unimodular lattices in $\R^{n+1}$, it was shown that if $\psi:(0,1)\to\R^n$ is analytic and its image is not contained in a proper affine subspace of $\R^n$, then $\psi(s)$ is not $\DTA{\lambda}$ and not $\DTB{\lambda}$ for almost all $s\in (0,1)$ and all $\lambda<1$.

For the smooth curve case, Shi and Weiss~\cite{Shi-Weiss:2017} showed that almost any point on a $2$-nondegenerate $C^2$-curve in $\R^2$ is not $\DTA{\lambda}$ for any $\lambda<1$, by proving equidistribution of {\em averages\/} of $a(t)$-translates of the associated curve in $\SL(3,\R)/\SL(3,\Z)$.

The analyticity of $\psi$ in \cite{Shah:Dirichlet} is a technical assumption because of a fundamental limitation of the linearization technique used in the proof, as the $(C,\alpha)$-good property \cite{Klein+Mar:Annals98} of differentiable maps may not survive under composition by non-linear polynomial maps. To overcome this difficulty, as in \cite{Shah:SOn1-smooth} for $G=\SO^0(n,1)$, we would like to prove an equidistribution result for expanding translates of shrinking curves. In this article, we make an algebraic observation that allows us to express expanding translates of optimally shrinking curves as long polynomial trajectories. Then, we apply an earlier result of Shah~\cite{Shah:polynom} about equidistribution of long polynomial trajectories. Our final equidistribution result leads to the non-improvability of Dirichlet's approximation theorem for nonplanar manifolds as defined by Pyartli~\cite[\S2]{Pyartli69}. 

\begin{definition}[Nonplanar submanifold] \label{def:Pyartli} 

A $k$-times differentiable map $\zeta:I\to \mathcal{A}$, where $I$ is an open subset of $\R^1$ and $\mathcal{A}$ is a $k$-dimensional Euclidean affine space, is called {\em nonplanar\/} if for each $r\in I$ the derivative vectors $\zeta^{(1)}(r),\ldots,\zeta^{(k)}(r)$ are linearly independent. 

Let $\psi$ be a $n$-times differentiable map from an open set $\Omega\subset\R^d$ to $\R^n$. We say that the submanifold $(\Omega,\psi,\R^n)$ is {\em nonplanar\/} at $s\in \Omega$ if the tangent space $\cT=D\psi(s)(\R^d)$ at $\psi(s)$ has dimension $d$, and there exists a $(n-d+1)$-dimensional subspace $\cM$ in $\R^n$ such that the following holds: $\cT+\cM=\R^n$, $\dim(\cT\cap\cM)=1$, and a curve naturally parameterizing the one-dimensional submanifold given by the intersection of the affine subspace $\psi(s)+\cM$ and the submanifold $\psi(\Omega_0)$, for some neighborhood $\Omega_0$ of $s$ in $\R^d$, is nonplanar.    

The submanifold $(\Omega,\psi,\R^n)$ is called {\em nonplanar}, if it is at all $s\in\Omega$.
\end{definition}

By naturally parameterizing curve in the above definition, we mean the following: in view of the constant rank theorem, for some neighborhood $\Omega_0$ of $s$ in $\R^d$ the set $\psi(\Omega_0)\cap (\psi(s)+\cM)$ is a  one-dimensional submanifold, and we parameterize it by the curve $\zeta:(r_1,r_2)\to (\psi(s)+\cM)\cap \psi(\Omega_0)$ for some $r_1<0<r_2$ such that $\zeta(0)=\psi(s)$ and $\R\zeta^{(1)}(0)=\cT\cap \cM$. 

We note that for a $n$-times differentiable map $\psi:\Omega\subset\R^1\to \R^n$, the one-dimensional manifold  $(\Omega,\psi,\R^n)$ is nonplanar at $s\in\Omega$ if and only if $\psi^{(1)}(s),\ldots,\psi^{(n)}(s)$ are linearly independent; here we have $\cT=\R\psi^{(1)}(s)$, and we pick $\cM=\R^n$ and $\zeta(r)=\psi(s+r)$ in the above definition.

\begin{theorem} \label{thm:Dirichlet}
Let $(\Omega,\psi,\R^n)$ be a $(n+1)$-times  differentiable, nonplanar submanifold. Then given an infinite set $\cN\subset\N$, for Lebesgue a.e.\ $s\in\Omega\subset\R^d$ and $\mbz=\psi(s)$, for any $0<\lambda<1$, there exist infinitely many $N\in\cN$ (depending on $s$) such that there is no integral solution to \eqref{eq:DTA} and no integral solution to \eqref{eq:DTB}.

In particular, $\psi(s)$ is not $\DTA{\lambda}$ or $\DTB{\lambda}$ for a.e.\ $s\in \Omega$ and any $0<\lambda<1$. 
\end{theorem}

\subsubsection*{Comparing the notions of nondegeneracy: $l$-nondegenerate (after Kleinbock and Margulis) versus nonplanar (after Pyartli)}

After Kleinbock and Margulis~\cite[Theorem~A]{Klein+Mar:Annals98}, for any $l\geq 1$, any $C^l$-map $\psi:\Omega\to\R^n$, where $\Omega\subset\R^d$ is open, and any $s\in \Omega$, if the partial derivatives $\partial_{i_k}\cdots\partial_{i_1}\psi(s)\in\R^n$ for all $1\leq i_j\leq d$ and $1\leq k\leq l$ span $\R^n$, then we say that {\em $\psi$ is $l$-nondegenerate at $s$}. And $\psi$ as above is called {\em $l$-nondegenerate}, if it is $l$-nondegenerate at every $s\in\Omega$. 

By \cite[Lemma 5]{Pyartli69}, if a submanifold $(\Omega,\psi,\R^n)$ is nonplanar at a point $s\in\Omega$ as in Definition~\ref{def:Pyartli}, then $\psi$ is $n$-nondegenerate at $s$.

Conversely, let $\psi:\Omega\subset \R^d\to \R^n$ be a $n$-times differentiable map which is $l$-nondegenerate for some $l\geq n$. If $d=1$, then by \cite[Corollary~3.3]{Shah-Yang:2022}) there exists a countable closed set $Z\subset\Omega$ such that the one-dimensional manifold $(\Omega\setminus Z,\psi,\R^n)$ is nonplanar. For $d>1$, if $\psi$ is an immersion and an analytic map, then one can show that there exists a closed subset $Z$ of $\Omega$ contained in a union of countably many $(d-1)$-dimensional analytic submanifolds of $\R^d$, such that the submanifold $(\Omega\setminus Z,\psi,\R^n)$ is nonplanar; here $Z$ is Lebesgue null.

\bigskip
As shown by Kleinbock and Weiss~\cite{Kleinbock+Weiss:Dirichlet} and Shah~\cite[Section~2]{Shah:Dirichlet}, using Dani's correspondence, Theorem~\ref{thm:Dirichlet} can be derived as a consequence of Theorem~\ref{thm:main-manifold}, which is about equidistribution of expanding translates of measures on submanifolds immersed in a homogeneous space. To formulate our results in greater generality, we use the following definition. 

\begin{definition}[Projectively nonplanar map]
\label{def:non-de} 
A $n$-times differentiable curve $\rho:I\to \mathcal{V}$, where $I$ is an open subset of $\R$ and $\mathcal{V}$ is a $k$-dimensional subspace of $\R^{n+1}$, is called \emph{projectively nonplanar}\/ in $\mathcal{V}$ if for each $s\in I$, the vectors $\rho(s),\rho^{1}(s),\ldots,\rho^{k-1}(s)$ form a basis of $\mathcal{V}$.  

Let $n,d\in\N$ and $d\leq n$. Let $\phi$ be a $n$-times differentiable map from an open subset $\Omega$ of $\R^d$ to $\R^{n+1}$. We say that $\phi$ is \emph{projectively nonplanar}\/ at $s\in \Omega$ if the following conditions are satisfied: The tangent space $\cT:=D\phi(s)(\R^d)$ has dimension $d$, $\phi(s)\not\in \cT$, 
    there exists a $(n-d+2)$-dimensional subspace $\cL$ of $\R^{n+1}$ containing $\phi(s)$ such that $\cT+\cL=\R^{n+1}$, $\dim(\cT\cap\cL)=1$, and for some neighborhood $\Omega_0$ of $s$ in $\R^d$, the curve naturally parameterizing the one-dimensional submanifold $\phi(\Omega_0)\cap \cL$ is projectively nonplanar in $\cL$. 
    
We say that the map $\phi$ is \emph{projectively nonplanar}\/ if it is projectively nonplanar at all $s\in\Omega$.
\end{definition}

Again, by naturally parameterizing curve in the above definition, we mean the following: in view of the constant rank theorem, the set $\phi(\Omega_0)\cap \cL$ is a one-dimensional submanifold, and we parameterize it by a curve $\rho:(r_1,r_2)\to \phi(\Omega_0)\cap \cL$ for some $r_1<0<r_2$ such that $\rho(0)=\phi(s)$ and $\R\rho^{(1)}(0)=\cT\cap\cL$. 

In the special case of $d=1$, a curve $\phi:\Omega\to \R^{n+1}$ is projectively nonplanar at $s\in\Omega$ if and only if the vectors $\phi(s),\phi^{(1)}(s),\ldots,\phi^{(n)}(s)$ form a basis of $\R^{n+1}$; here we have $\cT=\R\phi^{(1)}(s)$, and we pick $\cL=\R^{n+1}$ and $\rho(r)=\phi(s+r)$ in Definition~\ref{def:non-de}.

\medskip
We note that a submanifold $(\Omega,\psi,\R^n)$ is nonplanar at $s$ according to Definition~\ref{def:Pyartli} if and only if the map $\phi$, defined by $\phi(x):=(1,\psi(x))\in\R\times\R^n\cong \R^{n+1}$ for all $x\in\Omega$, is projectively nonplanar at $s$. In view of the notation in the corresponding definitions, if we identify $\R^n$ with $\{0\}\times\R^n$, then $\cT\subset \R^n$, and $\cL=\R\phi(s)+\cM$, or $\cM=\cL\cap\R^n$.

\subsubsection*{Notation} Let $1\leq d\leq n$, and $G=\SL(n+1,\R)$. For $t>0$, let 
\[
a(t):=\diag(t^{n},t^{-1},\ldots,t^{-1})\in G.
\]
Let $\M(n+1,\R)$ denote the set of $(n+1)\times(n+1)$ real matrices. Let $\Omega\subset\R^d$ be open and $\Phi:\Omega\to G$ be a continuous map. Then for any $s\in \Omega$, 
\begin{equation} \label{eq:Phi}
a(t)\Phi(s)=t^nJ_0\Phi(s)+t^{-1}J_n\Phi(s),
\end{equation}
where $J_0=\diag(1,0,\ldots,0),\, J_{n}=\diag(0,1,\dots, 1)\in\M(n+1,\R)$. For any $g\in\M(n+1,\R)$, we identify $J_0g$ with the top row of $g$ which is realized as an element of $\R^{n+1}$. We define $\phi:\Omega\to\R^{n+1}$ by $\phi(s)=J_0\Phi(s)\in\R^{n+1}$ for all $s\in\Omega$. 

\begin{theorem} \label{thm:main-manifold}
Suppose the map $\phi$ as above is $(n+1)$-times differentiable and projectively nonplanar. Let $L$ be a Lie group containing $G$, $\Lambda$ a lattice in $L$, and let $x\in L/\Lambda$. Then there exists $E_x\subset \Omega$ which is contained in a countable union of $C^1$ submanifolds of $\Omega$ of dimension $d-1$ such that the following holds: For every $s\in \Omega\setminus E_x$, and any bounded open convex neighborhood  $C$ of $0$ in $\R^d$, and any $f\in C_c(L/\Lambda)$,
\begin{equation} \label{eq:local-limit}
\lim_{t\to\infty} \frac{1}{\vol(C)} \int_{C} f(a(t)\Phi(s+t\inv \eta)x)\,d\eta=\int_{\cl{Gx}} f\,d\mu_x,
\end{equation}
where $\vol(\cdot)$ denotes the Lebesgue measure on $\R^d$, and $\mu_x$ is the unique $G$-invariant probability measure on $\cl{Gx}$ whose support equals $\cl{Gx}$.

In particular, for any probability measure $\nu$ on $\Omega$ which is absolutely continuous with respect to the Lebesgue measure, and any $f\in C_c(L/\Lambda)$,
\begin{equation} \label{eq:global-limit}
\lim_{t\to\infty} \int_{\Omega} f(a(t)\Phi(\eta)x)\,d\nu(\eta)=\int_{\cl{Gx}} f\,d\mu_x.
\end{equation}
\end{theorem}

We remark that due to Ranter's orbit closure theorem~\cite{Ratner:orbit}, $\cl{Gx}=Fx$ is a finite volume homogeneous space of a closed Lie subgroup $F$ of $L$ containing $G$, and any $G$-invariant finite Borel measure on $Fx$ whose support equals $Fx$ is also $F$-invariant.

\subsubsection*{Deduction of Theorem~\ref{thm:Dirichlet} from Theorem~\ref{thm:main-manifold}} For this purpose, we will apply Dani's correspondence principle as described in \cite[\S2.1]{Kleinbock+Weiss:Dirichlet}. First, we embed $G$ into $L=G\times G$ via the map $\rho$ as follows~\cite[\S1.0.1]{Shah:Dirichlet}: Let $\{e_i:1\leq i\leq n+1\}$ denote the standard basis of $\R^{n+1}$ and $\mathfrak{w}$ be the matrix such that $\mathfrak{w}e_i=e_{n-i+2}$ for all $i$. Let $\rho(g)=(g,\mathfrak{w}(^t\!g^{-1})\mathfrak{w}^{-1})$ for all $g\in G$. Then $\rho: G\to L$ is an injective homomorphism. Then as in \cite[\S2]{Shah:Dirichlet}, we can derive Theorem~\ref{thm:Dirichlet} from  \eqref{eq:global-limit} for $\Phi(s)=\bigl(\begin{smallmatrix} 1 & \psi(s)\\ 0 & I_n\end{smallmatrix}\bigr)$, where $I_n$ denotes the $n\times n$ identity matrix, the lattice $\Lambda=\SL(n+1,\Z)\times\SL(n+1,\Z)$ in $L$, and $x=\Lambda$; here $Gx$ is closed. 

\bigskip
To justify \eqref{eq:global-limit} for differentiable maps, we need to prove the equidistribution of local expansion given by \eqref{eq:local-limit}.
Our proof of \eqref{eq:local-limit} is very different from the arguments of \cite{Shah:Dirichlet} for proving \eqref{eq:global-limit} for analytic maps. A new basic identity \eqref{eq:basic_identity} observed in this article allows us to describe the limiting distribution of expansion of shrinking pieces in the curve ($d=1$) case using equidistribution of long polynomial trajectories on homogeneous spaces~\cite{Shah:polynom}. 

\subsubsection*{Some earlier results and shrinking speed}
Let $n=d=1$, $G=L=\SL(2,\R)$, a lattice $\Lambda$ in $L$, and 
$\Phi(s)=\bigl(\begin{smallmatrix}1&s\\&1\end{smallmatrix}\bigr)$.  In  \cite{Hej:Hua00,Str:long}, it was proved that if $s\mapsto\Phi(s)x$ is a closed horocycle, then for any $f\in C_c(L/\Lambda)$, for any sequence $t_i\to\infty$, and any intervals $[\alpha_i,\beta_i]\subset \R$ such that for some $\delta>0$, $\beta_i-\alpha_i\geq t^{-1+\delta}$ for all $i$, then 
\[
\lim_{i\to\infty} \frac{1}{\beta_i-\alpha_i} \int_{\alpha_i}^{\beta_i} f(a(t_i)\Phi(\eta)x)\,d\eta=\int_{L/\Lambda} f\,d\mu_{L/\Lambda}.
\]
Here the shrinking speed is slower compared to \eqref{eq:local-limit}, but the equidistribution occurs for shrinking around every $s$; that is, for some $\delta>0$ we let $\alpha_i=s-t_i^{-1+\delta}$  and $\beta_i=s+t_i^{-1+\delta}$ for each $i$. Later in Proposition~\ref{prop:E-dense} we will see that for $\Lambda=\SL(2,\Z)$ and $x=e\Lambda$,  \eqref{eq:local-limit} fails to hold for all rational $s$. In this case, the shrinking speed of $t^{-1}$ is indeed optimal as noted in \cite[\S2]{Hej:Hua00} and \cite[Page 509]{Str:long}. 

For the horospherical case of $L=G$, and $\psi(s)=s$ for all $s\in\R^n$, as in \cite{KM-exp,KM-eff} using the exponential mixing, one can deduce equidistribution for expanding translates by $a_t$ for sufficiently slowly shrinking horospherical balls. The proofs of \cite[Lemma~16]{Gorodnik:lattice} and \cite[Theorem~20]{Gorodnik:frames}, which use Ratner's theorem and linearization technique, yield the equidistribution of expanding translates of sufficiently slowly shrinking horospherical balls for the subgroup action on possibly larger homogeneous spaces $L/\Lambda$. Thus, the shrinking speed of $t^{-1}$ in \eqref{eq:local-limit} is faster than the previous results, but it may not be optimal.

\subsubsection*{Organization of the article}
We establish the basic identity mentioned above in \S\ref{sec:nondeg}.
In \S\ref{sec:curve}, we  combine the result on limiting distributions of polynomial trajectories with the basic identity to obtain an algebraic description of the limiting distribution of the stretching translates of the shrinking segments of the curve ($d=1$) 
around any given point $\Phi(s)x$ in $\Phi(\Omega)x$ (Theorem~\ref{thm:main-local}). In \S\ref{sec:d-shrinking}, we will derive the analogous result for shrinking balls around any given point in the submanifold  (Theorem~\ref{thm:d-shrinking}). For this purpose, we will fiber the shrinking balls into shrinking projectively nonplanar curve segments using a twisting trick due to Pyartli~\cite{Pyartli69}. A point $s\in\Omega$ is called exceptional if 
the limiting distribution of expanding translates of the shrinking balls in $\Phi(\Omega)x$ about the point $\Phi(s)x$ is not $G$-invariant. In \S\ref{sec:Ex}, we will obtain a geometric description of the set of exceptional points (Proposition~\ref{prop:E-H}) and prove that it is contained in a countable union of submanifolds of dimension $d-1$ which have zero Lebesgue measure (Proposition~\ref{prop:null}). Finally, we will show that in some standard examples, the exceptional points are dense in $\Omega$ (Proposition~\ref{prop:E-dense}).

\section{Basic identity} \label{sec:nondeg}
 The main new ingredient in the proof of 
Theorem~\ref{thm:main-manifold} is the following:

\begin{lemma}[Basic Identity] \label{lemma:basic-identity}
Let $d=1$, $\Omega\subset\R$ open, $\Phi:\Omega\to G$ a continuous map, and $s\in\Omega$ be such that the map $\phi:=J_0\Phi:\Omega\to\R^{n+1}$ is $(n+1)$-times  differentiable and projectively nonplanar at $s$. Then there exists a nilpotent matrix $B_s\in \M(n+1,\R)$ of rank $n$ such that for any $t\neq 0$ with $s+t^{-1}\in \Omega$, we have
\begin{equation} \label{eq:basic-identity}
a(\abs{t})\Phi(s+t^{-1})=(I+o(t^{-1})t)\xi_s({\sigma})\bigl(I+\sum_{k=1}^{n} t^{k}B_{s}^{k}\bigr),
\end{equation}
where $I$ denotes the identity matrix, $\sigma=t/\abs{t}=\pm1$, $\xi_s(\pm 1)\in G$, and $o(t^{-1})\in \M(n+1,\R)$ is such that $o(t^{-1})t\to 0$ as $\abs{t}\to\infty$.
\end{lemma}

We note that $B_{s}^{n}\neq 0$ and $B_{s}^{{n+1}}=0$, so
\begin{equation} \label{eq:Ps}
 P_{s}(t):=(I-tB_{s})^{-1}=I+\sum_{k=1}^{n} t^{k}B_{s}^{k}\in \SL(n+1,\R)=G. 
\end{equation}

\begin{proof}
We want to find a nilpotent matrix $B_{s}\in \M(n+1,\R)$ such that 
 \begin{equation*}
     \lim_{\abs{t}\to \infty} a(\abs{t})\Phi(s+t^{-1})(I-tB_{s})\in G.
 \end{equation*}
 
Let $t\neq 0$ such that $s+t^{-1}\in \Omega$. In view of \eqref{eq:Phi}, by Taylor's expansion, 
\[
J_0\Phi(s+t^{-1})=\phi(s+t^{-1})=\sum_{k=0}^{n+1} \frac{\phi^{(k)}(s)}{k!}t^{-k}+o(t^{-(n+1)}).
\]
For any $B_{s}\in \M(n+1,\R)$ and $\sigma=t/\abs{t}=\pm1$, we have
\begin{align}
    &a(\abs{t})J_0\Phi(s+t^{-1})(I-tB_{s}) 
    = \abs{t}^n\phi(s+t^{-1})(I-tB_{s}) \nonumber\\
    &=\sigma^n\Bigl(\bigl(\sum_{k=0}^{n+1} \frac{\phi^{(k)}(s)}{k!}t^{n-k}\bigr)+ o(t^{-1})\Bigr)
    (I-tB_{s})\nonumber\\
    &=\sigma^n
    \Bigl(-\phi(s)B_{s}t^{n+1}+\sum_{k=1}^{n}\bigl(\frac{\phi^{(k-1)}(s)}{(k-1)!}-\frac{\phi^{(k)}(s)}{k!}B_s\bigr) t^{n-k+1}\Bigr) 
    \nonumber\\ 
    &\quad +\sigma^n\xi_{s,1} + o(t^{-1})t, \label{eq:basic-phi}
\end{align}
where 
\begin{equation} \label{eq:xi1}
    \xi_{s,1}=
    \frac{\phi^{(n)}(s)}{n!}-\frac{\phi^{(n+1)}(s)}{(n+1)!}B_{s}.
\end{equation}
We want to choose $B_{s}$ such that all the coefficients of positive powers of $t$ vanish in \eqref{eq:basic-phi}; in other words, we want
\begin{equation} \label{eq:Bs}
    \phi(s)B_{s}=0 \text{ and } \frac{\phi^{(k)}(s)}{k!}  B_{s} = \frac{\phi^{(k-1)}(s)}{(k-1)!} \text{ for $1\leq k\leq n$}. 
\end{equation}

By our assumption, $\{\phi^{(k)}(s)/{k!}:0\leq k\leq n\}$ is a basis of $\R^{n+1}$. Therefore there exists a unique matrix $B_s$ such that \eqref{eq:Bs} holds. We note that with respect to the basis $\{\phi^{(k)}(s)/{k!}:0\leq k\leq n\}$ of $\R^{n+1}$ and the action from the right, $B_s$ is a strictly lower triangular nilpotent matrix of rank $n$. In particular, $\det(I-tB_s)=1$ for all $t\in\R$. 

Now  by \eqref{eq:basic-phi} and \eqref{eq:Bs}, we have the following key identity:
\begin{equation}
\label{eq:basic1}
a(\abs{t})J_0\Phi(s+t^{-1})(I-tB_{s})=\sigma^n \xi_{s,1} + o(t^{-1})t.
\end{equation}
Also, since $\Phi$ is differentiable at $s$,
\begin{align}
    a(\abs{t})J_n\Phi(s+t^{-1})(I-tB_{s})
&=\abs{t}^{-1}(J_n\Phi(s)+O(t^{-1}))(I-tB_{s}) \nonumber\\
&=\sigma \xi_{s,2}+O(t^{-1}), \label{eq:basic2}
\end{align} 
where 
\begin{equation} \label{eq:xi2}
\xi_{s,2}=-J_n\Phi(s)B_{s}.
\end{equation}

In view of \eqref{eq:Phi}, combining \eqref{eq:basic1} and \eqref{eq:basic2}: 
\begin{equation} \label{eq:basic_identity}
    a(\abs{t})\Phi(s+t^{-1})(I-tB_{s})=\xi_s(\sigma)+o(t^{-1})t,
\end{equation}
where in view of \eqref{eq:xi1} and \eqref{eq:xi2}, $\sigma=t/\abs{t}=\pm1$ and 
\begin{equation} \label{eq:xi}
\xi_s(\sigma)=\sigma^n\xi_{s,1}+\sigma\xi_{s,2}.
\end{equation}

Now \eqref{eq:basic-identity} follows from \eqref{eq:basic_identity}. Since the left hand side of \eqref{eq:basic_identity} belongs to $G$ for all $t$, by taking $t\to \pm\infty$, we get $\xi_s(\pm1)\in G$.
\end{proof}

The basic identity~\eqref{eq:basic-identity} was inspired by \cite[Proposition~A.0.1]{Pengyu-thesis}, which involved an intricate study of interactions of linear dynamics of intertwining copies of $\SL(2,\R)$ in $G$ and their Weyl group elements using \cite[Lemma~4.1]{Shah+Yang:Dirichlet}.

\section{Limiting distribution of polynomial trajectories and stretching translates of shrinking curves}

\label{sec:curve}

Our proof of Theorem~\ref{thm:main-manifold} for $d=1$ is based on Lemma~\ref{lemma:basic-identity} and the following result on limiting distribution of polynomial trajectories on homogeneous spaces, proved using Ratner's description~\cite{Ratner:measure} of ergodic invariant measures for unipotent flows. 

\subsubsection{Notation} \label{not:cHx} Let $L$ be a Lie group containing $G$ and $\Lambda$ be a lattice in $L$. Let $x\in L/\Lambda$. Let $\cH_x$ denote the collection of all connected Lie subgroups $H$ of $L$ such that $Hx$ is closed and admits an $H$-invariant probability measure, say $\mu_H$, which is ergodic with respect to an $\Ad_L$-unipotent one-parameter subgroup of $L$. Then $\cH_x$ is countable \cite[Theorem~1.1]{Ratner:measure}, \cite[Proposition 2.1]{Dani+Mar:limit}.  

\begin{theorem}[{Shah~\cite{Shah:polynom}}] \label{thm:shah:polynom}
Let $Q:\R\to G=\SL(n+1,\R)$ be a map whose each coordinate is a polynomial, and $Q(\R)$ contains the identity element $I$. Let $H$ be the smallest Lie subgroup of $L$ containing $Q(\R)$ such that $Hx$ is closed. Then $H\in \cH_x$, and for any $f\in C_c(L/\Lambda)$,
\[
\lim_{T\to\infty} \frac{1}{T}\int_{0}^T f(Q(t)x)\,dt =\int_{Hx} f\,d\mu_H.
\]
\end{theorem}

The following is its straightforward reformulation via change of variable. 

\begin{corollary} \label{cor:polynom}
Let the notation be as in Theorem~\ref{thm:shah:polynom}. Then for any $f\in C_c(L/\Lambda)$ and $c<d$,
\[
\lim_{T\to\infty} \frac{1}{d-c} \int_{c}^d f(Q(Ts)x)\,ds =\int_{Hx} f\, d\mu_H.
\]
\end{corollary}

From this result, we can deduce its following variation. 

\begin{corollary} \label{cor:polynom-2}
Let the notation be as in Theorem~\ref{thm:shah:polynom}. Let $\rho:\R\to G$ be a measurable map and $\nu$ be an absolutely continuous finite Borel measure on $\R$. Then for any $f\in C_c(L/\Lambda)$, 
\begin{equation} \label{eq:polynom-2}
\int_{\R} f(\rho(\eta)Q(T\eta)x)\,d\nu(\eta) \stackrel{T\to\infty}{\longrightarrow} 
\int_{\R} \Bigl[ \int_{Hx} f(\rho(\eta)y)\,\mu_H(y)\Bigr]\,d\nu(\eta).
\end{equation}
\end{corollary} 

\begin{proof} We can assume that $\abs{f}\leq 1$. And since $\nu$ is finite, due to Lusin's theorem, we can replace $\rho$ and $d\nu(\eta)/d\eta$ by continuous functions with compact support. 
Let $s\in\R$. Given $\epsilon>0$, there exists $\delta_0>0$ such that for all $\eta\in (s-\delta_0/2,s+\delta_0/2)$ and $y\in L/\Lambda$,
\[
\abs{(d\nu/d\eta)(\eta)-(d\nu/d\eta)(s)}\leq \epsilon \quad \text{and} \quad \abs{f(\rho(\eta)y)-f(\rho(s)y)}\leq \epsilon.
\]
Using these approximations and Corollary~\ref{cor:polynom}, for any $0<\delta<\delta_0$ there exists  $T_{s,\delta}\geq 1$ such that for all $T\geq T_{s,\delta}$,
\[
\Abs{\int_{s-\delta/2}^{s+\delta/2} f(\rho(\eta)Q(T\eta)x)\,d\nu(\eta) - 
\delta\cdot(d\nu/d\eta)(s)\cdot \int_{Hx} f(\rho(s)y)\,d\mu_H(y)}\leq 2\epsilon\delta.
\]
From this \eqref{eq:polynom-2} follows.
\end{proof}

\begin{theorem} \label{thm:main-local}
Let $d=1$ and the notation be as in Theorem~\ref{thm:main-manifold} and Notation~\ref{not:cHx}. Let $s\in \Omega$. Then there exists $H_s\in\cH_x$ such that the following holds: Let $\nu$ be an absolutely continuous finite Borel measure on $\R$.  Then for any $f\in C_c(L/\Lambda)$,
\begin{align}
&\lim_{t\to\infty} \int_{\R} f(a(t)\Phi(s+\eta t^{-1})x)\,d\nu(\eta) \nonumber\\
& =\int_{\R} \Big[\int_{H_sx} f(a(\abs{\eta})\xi_{s}(\sign(\eta))y)\,d\mu_{H_s}(y)\Bigr]d\nu(\eta), \label{eq:main-local}
\end{align}
where $\sign(\eta)=\eta/\abs{\eta}=\pm1$ and $\xi_s(\pm1)\in G$ are given by \eqref{eq:xi}. 

Moreover if $H_s\supset G$, then $\cl{Gx}=H_sx$, $\mu_x=\mu_{H_s}$, and 
\[
\lim_{t\to\infty} \int_{\R} f(a(t)\Phi(s+\eta t^{-1})x)\,d\nu(\eta)=\nu(\R)\cdot\int_{\cl{Gx}} f \,d\mu_{x}.
\]
\end{theorem}

\begin{proof}
Let $\eta\neq 0$. For $t\gg 1$, writing $h=\eta\inv t$, by \eqref{eq:basic-identity} and \eqref{eq:Ps},
\begin{align*}
    a(t)\Phi(s +\eta t^{-1})x
    &=a(\abs{\eta})a(\abs{h})\Phi(s+h^{-1})x\\
    &=a(\abs{\eta})(I+o(h^{-1})h)\xi_s({\sign(\eta)})P_s(h)x\\
    &=(I+\abs{\eta}^{-(n+1)}o(h^{-1})h))a(\abs{\eta})\xi_s({\sign(\eta)})P_s(h)x\\
    &=(I+\abs{\eta}^{-(n+1)}o(t^{-1})t)a(\abs{\eta})\xi_s(\sign(\eta))P_s(t\eta^{-1})x.
\end{align*}

Since $f$ is bounded, we can ignore the integration over a small neighborhood of $0$, outside which $\abs{\eta}^{-(n+1)}o(t^{-1})t$ is close to $0$ uniformly for all large $t$. So by uniform continuity of $f$ we can ignore the factor $(I+\abs{\eta}^{-(n+1)}o(t^{-1})t)$, and hence
\begin{align}
&\lim_{t\to\infty} \int_{\R} f(a(t)\Phi(s+\eta t^{-1})x)\,d\nu(\eta) \nonumber\\ 
&=\lim_{t\to\infty} \int_{\R} f(a(\abs{\eta})\xi_s({\sign(\eta)})P_s(t\eta^{-1})x)\,d\nu(\eta).  \label{eq:int_main}
\end{align}

By \eqref{eq:Ps}, $P_s(0)=I$. Let $H_s\in\cH_x$ be the smallest subgroup containing $P_s(\R)$. Applying Corollary~\ref{cor:polynom-2} to the image of $\nu$ on $\R$ under the map $\eta\mapsto \eta\inv$, from \eqref{eq:int_main} we obtain \eqref{eq:main-local}. 
\end{proof}

\begin{remark} \label{rem:Hs-Bs}
We note that the subgroup $H_s$ as in \eqref{eq:main-local} of Theorem~\ref{thm:main-local} is the smallest Lie subgroup of $G$ such that the orbit $H_sx$ is closed and its Lie algebra contains $\{B_s^k:1\leq k\leq n\}$. To verify this, note that $P_s(\R)\subset H_s$ and $B_s^{n+1}=0$. Therefore
in view of \eqref{eq:Ps}, 
\[
\Lie(H_s)\ni\log(P_s(t))=\log((I-tB_s)^{-1})=-\log((I-tB_s))=\sum_{k=1}^n t^kB_s^k/k
\]
for all $t$ in some neighborhood of $0$ in $\R$.  Therefore, by the taking $k$-th derivative at $0$, we get $B_s^k\in\Lie(H_s)$ for all $1\leq k\leq n$. 
\end{remark}

\section{Stretching translates of shrinking submanifolds}

\label{sec:d-shrinking}

In this section, we will obtain the analog of Theorem~\ref{thm:main-local} for $d\geq 2$. 

\subsubsection*{Notation} 
Let $n\geq d\geq 2$, $\Omega$ be an open subset of $\R^d$, $G=\SL(n+1,\R)$, and let $\Phi:\Omega\to G$ be a continuous map. Throughout this section we fix $s\in\Omega$, and suppose that $\phi:=J_0\Phi:\Omega \to \R^{n+1}$ is $(n+1)$-times differentiable and projectively nonplanar at $s$ (Definition~\ref{def:non-de}). So the derivative $D\phi(s):\R^d\to \R^{n+1}$ of $\phi$ at $s$ is injective.  Let $\cT:=D\phi(s)(\R^d)\subset\R^{n+1}$. Pick an ordered basis  $e_1,\ldots,e_d$ of $\cT$. Pick an inner product on $\cT$, and let $\SOd\cong\SO(d)$ denote the special orthogonal group acting on $\cT$.

\begin{theorem} \label{thm:d-shrinking}
There exists a rational function $\xi_s:\SOd\to G$ such that the following holds. Let $L$ be a Lie group containing $G$, $\Lambda$ be a lattice in $L$, and $x\in L/\Lambda$. Then there exists a closed subgroup $H_s$ of $L$ such that $H_sx$ is closed and admits an $H_s$-invariant probability measure, say $\mu_{H_s}$, and for any open bounded convex neighbourhood ${C}$ of $0$ in $\R^d$ and any $f\in C_c(L/\Lambda)$, 
\begin{align} 
&\lim_{t\to\infty}\frac{1}{\vol(C)}\int_{{C}} f(a(t)\Phi(s+t^{-1}\eta)x)\,d\eta \nonumber\\
 =& \int_{g\in \SOd} \int_{0}^{r_g}  \left[\int_{H_sx} f(a(r)\xi_s(g)y)d\,\mu_{H_s}(y)\right]r^{d-1}\,dr\,dg,
\label{eq:shrinking-equi-d}
\end{align}
where $r_g:=\sup\{r\geq 0:rge_1\in D\phi(s)(C)\}$,
and $dg$ is the Haar integral on $\SOd$ such that $\int_{\SOd} d^{-1}(r_g)^d\,dg=1$.

Moreover if $H_s\supset G$, then $\cl{Gx}=H_sx$, $\mu_x=\mu_{H_s}$, and
\begin{equation} \label{eq:HsG}
\lim_{t\to\infty}\frac{1}{\vol(C)}\int_{{C}} f(a(t)\Phi(s+t^{-1}\eta)x)\,d\eta = \int_{\cl{Gx}} f\,d\mu_x.
\end{equation}
\end{theorem}

\subsection*{Realizing the manifold as a graph over a tangent} Let $\cL_1$ be a subspace of $\R^{n+1}$ complementary to $\cT$; that is, $\cT\oplus \cL_1=\R^{n+1}$. Since $D\phi(0):\R^d\to \cT\subset \R^{n+1}$ has rank $d$, by the constant rank theorem, there exist open neighborhoods $\Omega_{\cT}$ of $0$ in $\cT$ and $\Omega_1$ of $s$ in $\R^d$, and a $(n+1)$-diffeomorphism $\Psi:\Omega_{\cT}\to \Omega_1$ and a $(n+1)$-times differentiable map $F:\Omega_{\cT}\to \cL_1$ such that $\Psi(0)=s$, $D\Psi(0)=D\phi(s)^{-1}$, and
\begin{equation} \label{eq:Psi}
    \phi(\Psi(\eta))=\phi(s)+\eta+F(\eta), \quad \forall \eta\in \Omega_{\cT}. 
\end{equation}
In particular, $F(0)=0$ and $DF(0)=0$.

Fix an open bounded convex neighborhood ${C}$ of $0$ in $\R^d$. Let $C_1=D\phi(s)({C})$, which is contained in $\cT$. Then for any $f\in C_c(L/\Lambda)$, 
\begin{align} 
&\lim_{t\to\infty}\frac{1}{\vol({C})}\int_{{\kappa\in C}} f(a(t)\Phi(s+t^{-1}\kappa)x)\,d\kappa, \nonumber\\
&\quad \text{changing the variable $\kappa\in\R^d$ to $\eta\in\cT$ such that 
$s+t\inv\kappa=\Psi(t\inv\eta)$,}\nonumber\\
&=\lim_{t\to\infty}\frac{1}{\vol({C})}\int_{\eta\in t\Psi\inv(s+t\inv{C})} f(a(t)\Phi(\Psi(t\inv\eta))x)\cdot\abs{\det(D\Psi(t^{-1}\eta))}\,d\eta \nonumber\\
&=\lim_{t\to\infty}\frac{1}{\vol(C_1)}
\int_{{C_1}} f(a(t)\Phi(\Psi(t^{-1}\eta))x)\,d\eta, \label{eq:B}
\end{align}
if any of the limits exist. Because since
\[
\eta=t\Psi\inv(s+t\inv \kappa)=D\Psi(0)\inv(\kappa)+O(t^{-2})t=D\phi(s)(\kappa)+O(t\inv),
\]
$\lim_{t\to\infty} \vol(t\Psi\inv (s+t\inv {C})\Delta C_1)=0$,  
and   $\vol({C})=\abs{\det(D\Psi(0))}\vol({C_1})$.

\subsection*{Nonplanar curves on the manifold via Pyartli's twisting} Since $\phi$ is projectively nonplanar at $s$, by Definition~\ref{def:non-de}, $\phi(s)\not\in\cT$ and we can pick a subspace $\cL$ of $\R^{n+1}$ containing $\phi(s)$ and of dimension $n-d+2$ such that $\cT+\cL=\R^{n+1}$ and $\dim(\cT\cap\cL)=1$, and we pick a neighborhood $\Omega_0$ of $s$ in $\R^d$, $r_1<0<r_2$, and a curve $\rho:(r_1,r_2)\to \cL$ parameterize the one-dimensional submanifold $\phi(\Omega_0)\cap \cL$ with $\rho(0)=\phi(s)$ and $\R\rho^{(1)}(0)=\cT\cap\cL$ such that the vectors $\rho(0),\rho^{(1)}(0),\ldots,\rho^{(n-d+1)}(0)$ form a basis of $\cL$. 

Let $\cL_1$ be the span of $\{\rho(0),\rho^{(2)}(0),\ldots,\rho^{(n-d+1)}(0)\}$. Then $\cT\oplus \cL_1=\R^{n+1}$, and we consider \eqref{eq:Psi} with respect to this choice of the subspace $\cL_1$. 

For any $0\neq w\in\cT$, the curve $\rho_w:I_w\to \R w+\cL_1$ given by
\begin{equation} \label{eq:rw}
\rho_w(r)=\phi(\Psi(rw)))=\phi(s)+rw+F(rw),\,\forall r\in I_w,
\end{equation}
parameterized the one-dimensional submanifold $\phi(\Omega_1)\cap (\R w+\cL_1)$, $\rho_w(0)=\phi(s)$, and $\rho_w^{(1)}(0)=w$, where $I_w=\{r\in\R: rw\in\Omega_{\cT}\}$. 

We recall that to say $\rho_w$ is projectively nonplanar at $0$ means that the set 
\[
\{\rho_w^{(i)}(0):0\leq i\leq n-d+1\}
\]
is a basis of $\R w \oplus \cL_1$; or equivalently, it consists of linearly independent vectors.

Suppose $v\in\cT\cap\cL\setminus\{0\}$. Then $\R v+\cL_1=\cL$. So $\rho_v$ and the curve $\rho$, as above, both parameterize the one-dimensional submanifold $\phi(\Omega_0\cap\Omega_1)\cap \cL$, $\rho_v(0)=\rho(0)=\phi(s)$, and $\rho_v^{(1)}(0)=v\neq 0$. Therefore $\rho_v=\rho\circ \eta$, where $\eta$ is a $C^{n}$-diffeomorphism from a neighborhood of $0$ to a neighborhood of $0$ fixing $0$. Since $\rho$ is projectively nonplanar at $0$, the same holds for $\rho_v$. So we pick $g_1\in\SOd$ such that $g_1e_1\in\cT\cap\cL$. Then the curve $\rho_{g_1e_1}$ is projectively nonplanar at $0$. 

Our goal is to `radially fiber' a neighborhood of $\phi(s)$ in $\phi(\Omega)$ by projectively nonplanar curves passing through $\phi(s)$. We will achieve this by twisting a projectively nonplanar curve like $\rho_{w}\subset \R w+\cL_1$; the added twist to $\rho_w$ is such a high degree that it does not affect the initial $(n-d+1)$-derivatives at $0$, and its higher derivatives span a subspace of $\R^{n+1}$ complementary to $\R w +\cL_1$. This twisting trick due to Pyartli is carried out below. 

Let $\gamma:\R\to \cT$ be a curve given by 
\begin{equation*} 
\gamma(r)=re_1+\sum_{i=2}^d r^{n-d+i} e_i\in \cT,\ \forall r\in\R.
\end{equation*}
For any $g\in \SOd$, define $\zeta_{g\gamma}:(-r_0,r_0)\to \R^{n+1}$ by
\begin{align}
     \zeta_{g\gamma}(r)=\phi(\Psi(g\gamma(r))) =\phi(s)+g\gamma(r)+ F(g\gamma(r)).
    \label{eq:Pyartli}
\end{align}
The next lemma shows that $\zeta_{g\gamma}$ is a desired twist of the curve $\rho_{ge_1}$.

\begin{remark} \label{rem:poly}
Let $0\leq k\leq n$. It is straightforward to verify that $g\mapsto (F\circ(g\gamma))^{k}(0)$ is a polynomial map of degree at most $k$ in coordinates of $g$. Hence $g\mapsto \zeta_{g\gamma}^{(k)}(0)$ is a polynomial map in coordinates of $g$.
\end{remark}

\begin{lemma}[{\cite[Lemma~5]{Pyartli69}}] \label{lemma:twisting}
Let $g\in\SOd$ be such that the curve $\rho_{ge_1}$, as defined in \eqref{eq:rw}, is projectively nonplanar at $0$. Then the map $\zeta_{g\gamma}$, defined by \eqref{eq:Pyartli},
is projectively nonplanar at $0$.
\end{lemma}

\begin{proof} It is straightforward to verify that
\begin{align}
    \zeta_{g\gamma}^{(k)}(0)&=\rho^{(k)}_{ge_1}(0)\in \R ge_1\oplus \cL_1 \text{, for } 0\leq k\leq (n-d+1), \nonumber\\
    \zeta_{g\gamma}^{(n-d+i)}(0)&=(n-d+i)!\cdot ge_i+(F\circ(g\gamma))^{(n-d+i)}(0)
    \nonumber 
    \\
    &=(n-d+i)!\cdot ge_i \in\cT \text{ modulo }\cL_1 \text{, for } 2\leq i\leq d. \label{eq:zeta-der}
\end{align}
Therefore the $\R$-span of $\{\zeta_{g\gamma}^{(k)}(0):0\leq k\leq n\}$ equals 
$\cL_1\oplus \sum_{i=1}^d\R ge_i=\R^{n+1}$.
\end{proof}

\subsection*{Polar fibering}
For $t\geq 1$, let $T_t:\SOd\times [0,\infty)\to \cT$ be given by 
\begin{equation} \label{eq:gamma-r}
    T_t(g,r)=tg\gamma(t\inv r)=g\cdot t\gamma(t\inv r)=g\cdot(re_1+\sum_{i=2}^d t^{-(n-d+i-1)} r^{n-d+i} e_i).
\end{equation}
We recall that $2\leq d\leq n$. Let $dg$ denote a Haar integral on $\SOd$. For a fixed $r>0$, under the map $g\mapsto T_t(g,r)$, the Haar measure on $\SOd$ projects to a rotation invariant measure on the sphere of radius $\norm{t\gamma(t\inv r)}$ in $\cT$ centered at $0$, where $\norm{\cdot}$ denotes the norm with respect to the chosen inner product on $\cT$. Then the image of the integral $\mathrm{d}g\times \norm{t\gamma(t\inv r)}^{d-1}\mathrm{d}\norm{t\gamma(t\inv r)}$ under the map $T_t$ equals to a nonzero multiple of the Lebesgue integral on $\cT$.

Let 
$r_{g,t}=\sup\{r\geq 0: T_t(g,r)\in C_1\}$.
Now $T_t(g,r)=rge_1+O(t\inv)$ uniformly in $g$ and bounded $r$. Therefore $r_{g,t}=r_g+O(t\inv)$, where
\[
r_g=\sup\{r\geq 0:rge_1\in C_1\}. 
\]
By \eqref{eq:gamma-r},
$
\frac{\norm{t\gamma(t\inv r)}^{d-1}}{r^{d-1}}\cdot \frac{d}{dr}\norm{t\gamma(t\inv r)}=1+O(t\inv r)^2.
$
Therefore continuing \eqref{eq:B}, by the change of variable $\eta=T_t(g,r)$,

\begin{align}
    &\lim_{t\to\infty}\frac{1}{\vol(C_1)}\int_{{C_1}} f(a(t)\Phi(\Psi(t^{-1}\eta))x))\,d\eta \nonumber\\
    =&\lim_{t\to\infty} \int\limits_{g\in \SOd}
    \left[
    \int\limits_{0}^{r_{g,t}} 
    f(a(t)\Phi(\Psi(t\inv T_t(g,r)))x) \cdot \norm{t\gamma(t\inv r)}^{d-1}\,d(\norm{t\gamma(t\inv r)})
    \right]
    dg \nonumber \\
    =&\lim_{t\to\infty} \int\limits_{g\in \SOd}
    \left[
    \int\limits_{0}^{r_g} 
    f(a(t)\Phi(\Psi(g\gamma(t\inv r)))x)r^{d-1}dr
    \right]
    dg,
    \label{eq:T}
\end{align}

where for each $t$ the Haar integral $dg$ on $\SOd$ is normalized such that the integral of the expression equals $1$ for the constant function $f=1$.

\subsection*{Proof of Theorem~\ref{thm:d-shrinking}} 
In view of \eqref{eq:Pyartli} and \eqref{eq:T},
\begin{equation} \label{eq:zeta}
J_0\Phi(\Psi(g\gamma(r)))=\zeta_{g\gamma}(r)\in\R^{n+1}.
\end{equation}

Let $\{\te_k:0\leq k\leq n\}$ denote the standard basis of $\R^{n+1}$ consisting of row vectors. For $g\in \SOd$, let $M{}(g)\in\M(n+1,\R)$ be such that with respect to the right action $\R^{n+1}$, 
\begin{equation*} 
\te_k M{}(g)=\zeta_{g\gamma}^{(k)}(0)/k!, \ \forall\, 0\leq k\leq n.
\end{equation*}

By Remark~\ref{rem:poly}, the map $g\mapsto M(g)$ is a polynomial function in coordinates of $g$. Define
\[
Z_s:=\{g\in\SOd:\det(M(g))=0\}.
\]
Then $Z_s$ is an algebraic subvariety of $\SOd$. For $g\in \SOd$, $\zeta_{g\gamma}$ is projectively nonplanar at $0$ if and only if $g\not\in Z_s$.  

Pick $g_1\in\SOd$ such that $g_1e_1\in \cL\cap\cT\setminus\{0\}$. As observed earlier, the curve $\rho_{g_1e_1}$ is projectively nonplanar at $0$. Therefore by Lemma~\ref{lemma:twisting}, $\zeta_{g_1\gamma}$ is projectively nonplanar at $0$. Hence $g_1\not\in Z_s$. Therefore $Z_s$ is a proper algebraic subvariety of $\SOd$ and $\dim(\cT)=d\geq 2$. Hence $Z_s$ is Haar-null on $\SOd$.

Let $B$ denote the lower triangular matrix such that $\te_0B=0$ and $\te_kB=\te_{k-1}$ for $1\leq i\leq n$. Let $g\in \SOd\setminus Z_s$. Set $B{}(g)=M{}(g)\inv B M{}(g)$. Then 
\[
\zeta_{g\gamma}(0)B{}(g)=0 \text{ and } (\zeta^{(k)}_{g\gamma}(0)/k!)B{}(g)=\zeta^{(k-1)}_{g\gamma}(0)/(k-1)!,
\ \forall 1\leq k\leq n,
\]
as in \eqref{eq:Bs}. In view of \eqref{eq:xi}, let
\begin{equation*} 
\xi_s(g)=J_0(\zeta_{g\gamma}^{(n)}(0)/{n!}  - \zeta_{g\gamma}^{(n+1)}(0)/(n+1)!\cdot B{}(g)) - J_n\Phi(s)B{}(g).
\end{equation*}
Then by \eqref{eq:zeta} and \eqref{eq:basic_identity},
\begin{equation} \label{eq:basic-3}
a(t)\Phi(\Psi(g\gamma(t\inv)))
=(I+o(t\inv)t)\xi_s(g)(I-tB{}(g))^{-1}.
\end{equation}
In particular, $\xi_s(g)\in G$. As in  \eqref{eq:Ps}, 
\begin{equation} \label{eq:inv}
(I-tB{}(g))^{-1}=I+\sum_{k=1}^n t^k B{}(g)^k,\  \forall\, t\in\R.
\end{equation}

Let $\lief{}(g)$ be the $\R$-span of $\{B{}(g)^k:1\leq k\leq n\}$. Then one has
\begin{equation} \label{eq:fg}
    \lief{}(g)=M{}(g)\inv\lief M{}(g),
\end{equation}
where $\lief$ is the $\R$-span of $\{B^k:1\leq k\leq n\}$.

We fix $x\in L/\Lambda$. Let $H(g)\in\cH_x$ be the smallest Lie subgroup such that its Lie algebra contains $\lief(g)$. By Theorem~\ref{thm:main-local}, Remark~\ref{rem:Hs-Bs}, \eqref{eq:T}, \eqref{eq:basic-3}, and \eqref{eq:inv},
\begin{align}
&\lim_{t\to\infty} \int_{0}^{r_g} f(a(t)\Phi(\Psi(g\gamma(t\inv r)))x)r^{d-1}\,dr \nonumber\\
& =\int_0^{r_g} \Big[\int_{H(g)x} f(a(r)\xi_{s}(g)y)\,d\mu_{H(g)}(y)\Bigr]r^{d-1}dr. \label{eq:main-local-g}
\end{align}

\subsubsection*{Claim~1.} {\em There exists $g_0\in \SOd\setminus Z_s$ such that $H(g)\subset H(g_0)$,  $\forall g\in \SOd\setminus Z_s$.}

\subsubsection*{Proof of Claim~1.}
For any $H\in\cH_x$ and $g\in \SOd\setminus Z_s$, 
we have $H(g)\in\cH_x$, and 
\begin{align} 
H(g)\subset H\iff \lief(g)\subset\Lie(H) 
\iff M(g)\inv \lief M{}(g)\subset \Lie(H). \label{eq:Hs-2}
\end{align}

For any $H\in\cH_x$, we define
\[
Z_s(H)=\{g\in\SOd: \lief M(g)\subset M(g)\Lie(H)\}.
\]
Then by \eqref{eq:Hs-2},
\begin{equation} \label{eq:ZsH}
Z_s(H)\setminus Z_s=\{g\in \SOd\setminus Z_s: H(g)\subset H\}. 
\end{equation}
Since $g\mapsto M(g)$ is a polynomial map, $Z_s(H)$ is a Zariski closed subset of $\SOd$. For every $g\in\SOd\setminus Z_s$, we have $g\in Z_s(H(g))$. Since $\cH_x$ is countable, $\SOd\setminus Z_s$ is covered by a countable union of closed subsets $Z_s(H(g))$, where $g\in \SOd\setminus Z_s$. Since $\SOd\setminus Z_s$ is locally compact, it is of Baire's second category. So there exists $g_0\in\SOd\setminus Z_s$ such that $Z_s(H(g_0))$ contains a non-empty open subset of $\SOd$. Since $\dim(\cT)=d\geq 2$, any non-empty open subset of $\SOd$ is Zariski dense in $\SOd$. Therefore $Z_s(H(g_0))=\SOd$. So the Claim~1 follows from \eqref{eq:ZsH}.

\subsubsection*{Claim~2.} {\em Pick $g_0\in \SOd\setminus Z_s$, and define
\[
Z_{s,g_0}=\{g\in\SOd\setminus Z_s:\,H(g_0)\not\subset H(g)\}.
\] 
Then $Z_{s,g_0}$ is Haar-null on $\SOd$.}

\subsubsection*{Proof of Claim 2.}
Let $g\in Z_{s,g_0}$. Then by \eqref{eq:ZsH}, $g_0\not\in Z_s(H({g}))$ and $g\in Z_s(H({g}))$. So $Z_s(H({g}))$ is an algebraic subvariety of $\SOd$ of strictly lower dimension. Now $H(g)\in\cH_x$ and $\cH_x$ is countable. Therefore $Z_{s,g_0}$ is contained in a countable union of proper algebraic subvarieties of $\SOd$. Hence $Z_{s,g_0}$ is Haar-null on $\SOd$, proving Claim~2.

\bigskip
Now pick $g_0$ as in Claim~1. Put $H_s=H(g_0)$. By Claim~1, and the definition of $Z_{s,g_0}$,  we have $H(g)=H_s$ for all $g\in \SOd \setminus (Z_s\cup Z_{s,g_0})$. Continuing \eqref{eq:T}, using \eqref{eq:main-local-g},  since $Z_s\cup Z_{s,g_0}$ is Haar-null by Claim~2,  
\begin{align}
&\lim_{t\to\infty} \int_{g\in\SOd} 
    \left[
    \int_{0}^{r_g} 
    f((a(t)\Phi(\Psi(g\gamma(t\inv r))))x)r^{d-1}\,dr
    \right]
    dg \nonumber \\
    =&
    \int_{g\in \SOd\setminus (Z_{s,g_0}\cup Z_s)}
    \int_{0}^{r_g} \left[\int_{y\in H_sx} f(a(r)\xi_s(g)y)\,d\mu_{H_s}\right] r^{d-1}\, dr\,dg.
\end{align}
This completes the proof of Theorem~\ref{thm:d-shrinking}. \qed

\section{Equidistribution of translates of nonplanar manifolds}
\label{sec:Ex}
Let the notation be as in the statement of Theorem~\ref{thm:main-manifold}. In view of \eqref{eq:HsG} in Theorem~\ref{thm:d-shrinking}, we define the exceptional set
\begin{equation} \label{eq:def:Ex}
E_x=\{s\in \Omega: G\not\subset H_s\}.
\end{equation}

To describe $E_x$, we will use a crucial result from \cite{Shah:Dirichlet}, which is generalized in \cite{Yang2020} for arbitrary $G$. 
We begin with some notation and observations. Let
\begin{align*}
U&=\{u(\mbz):=\bigl(\begin{smallmatrix} 1 & \mbz\\ 0 & J_n\end{smallmatrix}\bigr): \mbz\in\R^n\}=\{g\in G: \lim_{t\to\infty} a(t)^{-1} g a(t)=I\}\\
U^-&=\{g\in G: \lim_{t\to\infty} a(t) g a(t)^{-1}=I\}\\
P^-&=\{g\in G:\cl{\{a(t)ga(t)\inv:t\geq 1\}}\text{ is compact}\}=Z_G(\{a(t):t>0\})U^-.
\end{align*}

Let $\{\te_{k}:0\leq k\leq n\}$ denote the standard basis of $\R^{n+1}$, which is identified with $J_0\M(n+1,\R)$, the space of top rows of matrices in $\M(n+1,\R)$. We identify $\R^n$ with $\Span\{\te_{k}:1\leq k\leq n\}$. 
Then $P^-$ is the stabilizer of the line $\R\cdot \te_0$ for the right action of $G$ on $\R^{n+1}$ and $\te_0 u(\mbz)=\te_{0}+\mbz\in\R^{n+1}$, $\forall\mbz
\in\R^{n}$. Therefore
\begin{equation} \label{eq:P-U}
P^-U=\{g\in G: g_{00}:=\langle \te_0g,\te_0\rangle\neq 0\}=\{g\in G:J_0g\notin \{0\}\times\R^n\},
\end{equation}
and it is a Zariski open dense neighborhood of the identity in $G$. 

For a finite dimensional representation $V$ of $G$, define
\begin{gather*}
    V^{+}=\{v\in V:\lim_{t\to\infty} a(t)\inv v=0\}, \quad V^{-}=\{v\in V:\lim_{t\to\infty} a(t) v=0\},\\
    V^0=\{v\in V: a(t)v=v,\ \forall t>0\}.
\end{gather*}
Then $V=V^+\oplus V^0\oplus V^-$, and let $\pi_+$, $\pi_0$, and $\pi_-$ denote the corresponding projections from $V$ onto $V^+$, $V^0$ and $V^-$, respectively. 

\begin{proposition} 
[{\cite[Corollary~4.4]{Shah:Dirichlet}}] \label{prop:basic}   Let $\cE\subset P^{-}U$ be such that $J_{0}\cE$ is not contained in the union of any $n$ proper subspaces of $\R^{n+1}$. Then for any finite dimensional representation $V$ of $G$ and a nonzero $v\in V$, if
\begin{equation} \label{eq:cE}
    gv\in V^0+V^-, \quad \forall g\in\cE,
\end{equation}
then $\pi_0(gv)\neq 0$ for all $g\in\cE$ and $Z_G(\{a(t):t>0\})$ fixes $\pi_0(gv)$.

\end{proposition}

\begin{proof}
For every $g\in P^{-}U$, there exists a unique $\bar g\in\R^{n}$ 
such that $P^{-}g=P^{-}u(\bar g)$, and $J_{0}g=g_{00}(\tilde e_0+\bar g)$. 
Suppose that $v\in V$ is such that \eqref{eq:cE} holds.
Since $P^{-}$ stabilizes $V^{0}+V^{-}$,
\begin{equation} \label{eq:bar_cE}
u({\bar g}) v\subset V^{0}+V^{-}, \quad \forall g\in \cE.
\end{equation} 

\subsubsection*{Claim~1}  {\em Let $h\in\cE$. Then for any proper subspaces $W_{k}$ of $\R^{n}$ for $1\leq k\leq n$,  
\[
\{\bar g - \bar h: g\in\cE\}\not\subset \cup_{k=1}^n W_{k}.
\]}

To prove this by contradiction, suppose $\{\bar g-\bar h:g\in \cE\}\subset \cup_{k=1}^n W_{k}$. For every $g\in P^{-}U$, $\bar g-\bar h=g_{00}\inv J_{0}g-h_{00}\inv J_{0}h$. Therefore $J_{0}\cE\subset \cup_{k=1}^n (W_{k}\oplus \R J_{0}h)$, which contradicts the choice of $\cE$, as $W_{k}\oplus \R J_{0}h$ is a proper subspace of $\R^{n+1}$ for each $k$.

\bigskip
Let $h\in\cE$. By \cite[Corollary~4.4]{Shah:Dirichlet}, \eqref{eq:bar_cE}, and Claim~1,  $\pi_{0}(u(\bar h)v)\neq 0$ and it is fixed by $Z_{G}(\{a(t):t>0\})$. 
Now $h=bu(\bar h)$ for some $b\in P^{-}$. So $b=zu^-$ for some $z\in Z_G(\{a(t):t>0\})$ and $u^-\in U^-$. Then for any $w\in V$, we have $\pi_0(zw)=z\pi_0(w)$, and for any $w\in V^0+V^-$, we have $u^-w\in V^0+V^-$ and $\pi_0(u^-w)=\pi_0(w)$. Since $h=zu^-u(\bar h)$ and $u(\bar h)v\in V^0+V^-$, we  have 
\[
\pi_0(hv)=z\pi_0(u^{-}\cdot u(\bar h)v)=z\pi_0(u(\bar h)v), 
\]
which is nonzero and fixed by $Z_G(\{a(t):t>0\})$.
\end{proof} 

\begin{proposition} \label{prop:E-H}
Let $H\in \cH_x$ be such that $G\not\subset H$. Let
\begin{equation} \label{eq:def-E_H}
E_H=\{s\in \Omega:\Phi(s)\in P^-U \text{ and } \,H_s\subset H\}.
\end{equation}
Then $\phi(E_H)$ is contained in a union of countably many proper subspaces of $\R^{n+1}$; we recall that $\phi(s)=J_0\Phi(s)$, $\forall s\in\Omega$. 
\end{proposition}

\begin{proof}
Let $F$ be the closure of the subgroup of $G$ generated by all unipotent elements of $G$ contained in $H$. Then $F\neq G$. Since $H$ is a connected Lie group, $F$ is a real algebraic subgroup of $G$. Since $F$ admits no nontrivial algebraic characters, we choose a finite-dimensional representation $V$ of $G$ with a vector $p_F\in V$ such that $F$ fixes $p_F$ and $V$ has no nonzero $G$-fixed vector. 

\subsubsection*{Claim 2}  {\em For any $s\in E_H$, 
    $\Phi(s)p_F\in V^0+V^-$.}

\bigskip
To see this, for any  $g\in \SOd$, since $\gamma(0)=0$ and $\Psi(0)=s$, we have
\begin{equation} \label{eq:pi+}
\lim_{t\to\infty} \pi_+(\Phi(\Psi(g\gamma(t\inv)))p_F)= \pi_+(\Phi(s)p_F).
\end{equation}
Now let $s\in E_H$ and $g\in\SOd\setminus (Z_{s}\cup Z_{s,g_0})$. By \eqref{eq:basic-3}, 
\begin{align}
a(t)\Phi(\Psi(g\gamma(t\inv)))p_F 
&=(I+o(t^{-1})t)\xi_s(g)(I-tB_s(g))\inv p_F \nonumber\\
&=(I+o(t^{-1})t)\xi_s(g)p_F, \label{eq:lim_T1}
\end{align}
because by \eqref{eq:inv}, \eqref{eq:fg} and \eqref{eq:Hs-2}, $(I-tB_s(g))\inv$ is a unipotent element of $G$ contained in $H(g)=H_s\subset H$, so it fixes $p_F$. Therefore from \eqref{eq:pi+} we conclude that $\pi_+(\Phi(s)p_F)=0$, otherwise the left hand side of \eqref{eq:lim_T1} diverges as $t\to\infty$. So Claim~2 holds. 

\subsubsection*{Claim 3} {\em For any $s\in E_H$, if there exists a sequence $\{s_i\}\subset E_H$ such that $s_i\to s$ and $\phi(s_i)\not\in \R \phi(s)$, $\forall\,i$, then  $\pi_0(\Phi(s)p_F)$ is fixed by $u(\mbz)\in U$ for some $\mbz\in\R^n\setminus\{0\}$.} 

\bigskip
Since $\phi(s_i)=J_0\Phi(s_i)$ and $J_0P^-\Phi(s)\subset\R \phi(s)$, we have
$\Phi(s_i)\not\in P^-\Phi(s)$ for all $i$. 
 Therefore  $\Phi(s_i)\Phi(s)^{-1}=b_iu(\mbz_i)$, where $b_i\to I$ in $P^-$ and $0\neq \mbz_i\to 0$ as $i\to\infty$. Let $t_i=\norm{\mbz_i}^{-1/(n+1)}$. After passing to a subsequence, as $i\to\infty$,
\begin{equation} \label{eq:u}
 a(t_i)\Phi(s_i)\Phi(s)^{-1}a(t_i)^{-1}=a(t_i)b_ia(t_i^{-1})\cdot u(\mbz_i/\norm{\mbz_i})\to u(\mbz), 
\end{equation}
for some $\mbz\in \R^n\setminus\{0\}$.
By Claim~2, $\Phi(s_i)p_F\in V^0+V^-$, so
\begin{equation*}
    \lim_{i\to\infty} a(t_i)\Phi(s_i)p_F=\lim_{i\to\infty} \pi_0(\Phi(s_i)p_F)=\pi_0(\Phi(s)p_F).
\end{equation*}
On the other hand by \eqref{eq:u}, as $i\to\infty$,
\begin{align*}
a(t_i)\Phi(s_i)p_F
=a(t_i)\Phi(s_i)\Phi(s)^{-1}a(t_i)\inv\cdot a(t_i)\Phi(s)p_F \nonumber
\to u(\mbz)\cdot \pi_0(\Phi(s)p_F). 
\end{align*}
Therefore $\pi_0(\Phi(s)p_F)$ is fixed by $u(\mbz)$. This proves Claim~3.

\subsubsection*{Claim 4} {\em If $\phi(E_H)$ is not contained in the union of any $n$ proper subspaces of $\R^{n+1}$, 
\begin{equation} \label{eq:stab}
\text{
 $\forall s\in E_H$, we have $\pi_0(\Phi(s)p_F)\neq 0$, and it is fixed by $Z_G(\{a(t):t>0\})$.}
\end{equation}}

This follows from Proposition~\ref{prop:basic} applied to $\cE=\Phi(E_H)$ and Claim~2.

\subsubsection*{Claim 5} {\em $\phi(E_H)$ is contained in the union of $n$ proper subspaces of $\R^{n+1}$, or for each $s\in E_H$, there exists a neighbourhood $\Omega_2$ of $s$ such that $\phi(E_H\cap\Omega_2)\subset \R\phi(s)$.}

\medskip
Suppose the claim fails to hold. Then the condition of Claim~4 holds, and we can pick some $s\in E_H$ such that the condition of Claim~3 holds for $s$. Then $\pi_0(\Phi(s)p_F)\neq 0$, and it is fixed by the subgroup generated by $u(\mbz)$ and $Z_G(\{a(t):t>0\})$. Since every nontrivial element of $U$ is conjugated to $u(\mbz)$ by an element of $Z_G(\{a(t):t>0\})$, we have that $\pi_0(\Phi(s)p_F)$ is fixed by $Z_G(\{a(t):t>0\})U$, which is a parabolic subgroup of $G$. So $0\neq \pi_0(\Phi(s)p_F)$ is fixed by $G$, a contradiction to our choice of $V$. This proves Claim~5, which implies the conclusion of the proposition.
\end{proof}

We need the following observation to obtain further information on the set $E_H$.

\begin{proposition} \label{prop:null} 
Let $\Omega\subset\R^d$ be an open set, and $\phi:\Omega\to\R^{n+1}$ be a $C^{l}$-map for some $l\geq 1$, such that for each $x\in\Omega$, the vectors $\phi(x)$ and $\partial_{i_k}\cdots\partial_{i_1}\phi(x)$ for all $i_j\in\{1,\ldots,d\}$ and $1\leq k\leq l$ span $\R^{n+1}$. Then for any proper subspace $W$ of $\R^{n+1}$, the set $\phi^{-1}(W)$ is contained in a union of countably many $C^1$-submanifolds of dimension $d-1$.
\end{proposition}

We will prove this using the following straightforward consequence of the implicit function theorem.

\begin{lemma}
\label{lema:implicit}
Let $\Omega\subset\R^d$ be open, and $\psi:\Omega\to \R$ be a $C^1$-map. Let 
\[
Y=\{x\in \Omega:\psi(x)=0,\,\exists i\in\{1,\ldots,d\},\, \partial_i\psi(x)\neq0\}.
\]
Then $Y$ is a $C^1$-submanifold of $\Omega$ of dimension $d-1$.
\end{lemma}

\begin{proof}
Let $y\in Y$. Pick $i\in \{1,\ldots,d\}$ such that $\partial_i\psi(y)\neq 0$. Without loss of generality, by permuting the coordinates, we assume that $i=d$. Let $(a,b)\in \R^{d-1}\times \R$ be such that $y=(a,b)$. Then by implicit function theorem there exist neighborhoods $V$ of $a$ in $\R^{d-1}$ and $W$ of $b$ in $\R$, and a $C^1$-map $f:V\to W$ such that 
\[
Y\cap (V\times W)=\{(v,w)\in V\times W: \psi((v,w))=0\}=\{(v,f(v)):v\in V\},
\]
which is a submanifold of $\R^d$ of dimension $d-1$.
\end{proof}

\begin{proof}[Proof of Proposition~\ref{prop:null}] 
Let $\ell$ be a nonzero linear functional on $\R^{n+1}$ such that $\ell(W)=0$. Let $\psi=\ell\circ \phi$. 
Then $\phi^{-1}(W)\subset X_0:=\{x\in \Omega:\psi(x)=0\}$. 

For $1\leq k\leq l$, let
\begin{align*}
X_k&=\{x\in X_0:\forall 1\leq j\leq k,\,\forall i_1,\ldots,i_j\in\{1,\ldots,d\},\,\partial_{i_j}\cdots\partial_{i_1}\psi(x)=0\}.
\end{align*}

Note that $\partial_{i_j}\cdots\partial_{i_1}\psi=\ell\circ\partial_{i_j}\cdots\partial_{i_1}\phi$. Since $\ker\ell$ is a proper subspace of $\R^{n+1}$, by our assumption on $\phi$ we conclude that $X_{l}=\emptyset$. Therefore,
\begin{equation} \label{eq:X0}
X_0=\bigcup_{k=1}^l X_{k-1}\setminus X_k.
\end{equation}

Let $1\leq k\leq l$, and $y\in Y_k:=X_{k-1}\setminus X_k$. Since $y\not\in X_{k}$, we pick $i_1,\ldots,i_{k}\in\{1,\ldots,d\}$ such that $\partial_{i_{k}}\cdots\partial_{i_1}\psi(z)\neq 0$ for all $z$ in an open set $\Omega_1$ containing $y$. Let $\psi_1:=\partial_{i_{(k-1)}}\cdots\partial_{i_1}\psi:\Omega_1\to\R$. Then $\psi_1$ is a $C^1$-map. 
For any $x\in Y_k\cap \Omega_1\subset X_{k-1}$, we have $\psi_1(x)=0$. Therefore by Lemma~\ref{lema:implicit} applied to $\psi_1$, the neighborhood $Y_k\cap \Omega_1$ of $y$ in $Y_k$ is contained in a $C^1$-submanifold of dimension $d-1$. Since $Y_k$ is covered by countably many such neighborhoods, the conclusion of the proposition follows from \eqref{eq:X0}. 
\end{proof}

\begin{proof}[Proof of Theorem~\ref{thm:main-manifold}] Let $s\in \Omega\setminus\phi\inv(\{0\}\times\R^n)$. Then $J_0\Phi(s)=\phi(s)\not\in \{0\}\times \R^n$. So by \eqref{eq:P-U}, $\Phi(s)\in P^-U$. Now suppose $s\in E_x\setminus\phi\inv(\{0\}\times\R^n)$. Then $\Phi(s)\in P^-U$ and $G\not\subset H_s$ by \eqref{eq:def:Ex}. Also $H_s\in\cH_x$. So $s\in E_{H_s}$ by \eqref{eq:def-E_H}. Therefore
\begin{equation*} 
E_x\subset \phi\inv(\{0\}\times\R^n)\bigcup \cup\{E_H:H\in\cH_x \text{ and } H\not\supset G\}.
\end{equation*}
Since $\cH_x$ is countable,  by Proposition~\ref{prop:E-H}, $E_x$ is contained in a countable union of sets of the form $\phi\inv(W)$, where $W$ is a proper linear subspace of $\R^{n+1}$. Due to Lemma~\ref{lemma:twisting}, $\phi$ satisfies the condition of Proposition~\ref{prop:null} for $\ell=n$, and hence $E_x$ is contained in a countable union of $C^1$-submanifolds of dimension $d-1$, and in particular, its measure is zero. Let $s\in\Omega\setminus E_x$. Then $H_s\supset G$. Therefore by Theorem~\ref{thm:d-shrinking}, we get  \eqref{eq:HsG}, which is the same as  \eqref{eq:local-limit}. Now \eqref{eq:global-limit} can be deduced from \eqref{eq:local-limit} using the Lebesgue points of $\nu$ and convergence in measure.
\end{proof}

Next, we show that the exceptional set is dense in some standard examples.

\begin{proposition} \label{prop:E-dense}
Let $L=G=\SL(n+1,\R)$ and $\Lambda=\SL(n+1,\Z)$. Let $\phi:\R\to \R^{n+1}$ be a polynomial map with coefficients in $\Q$ such that its image is not contained in a proper subspace of $\R^{n+1}$. Then $\phi$ is projectively nonplanar on $\R\setminus Z$, where $Z$ is a finite subset of $\R$. Let $\Phi:\R\setminus Z\to G$ be a continuous map such that $J_0\Phi(s)=\phi(s)$ for all $s\in\R\setminus Z$. Let $x\in\SL(n+1,\Q)/\Lambda\subset L/\Lambda$. Then $E_x\supset \Q\setminus Z$. 
\end{proposition}

\begin{proof}
For any $s\in \R$, let $D(s)\in M(n+1,\R)$ be the matrix whose $i$-th row is $\phi^{(i-1)}(s)\in\R^{n+1}$ for $1\leq i\leq n+1$. Let $Z$ be the set of roots of polynomial $s\mapsto \det(D(s))$. Then $\phi$ is non-degenerate at all $s\in \R\setminus Z$. If $Z=\R$, then $\phi(\R)$ is contained in a proper subspace of $\R$; see \cite[Corollary~3.3]{Shah-Yang:2022}. Therefore, $Z$ is finite.

Let $s\in \Q\setminus Z$. Since $\phi^{(k)}(s)/k!\in \Q^{n+1}$ for $0\leq k\leq n$, by \eqref{eq:Bs} we have that the nilpotent matrix $B_s\in M({n+1},\Q)$. So the $\R$-span of $\{B_s^k:1\leq k\leq n\}$ is an abelian Lie subalgebra of $\Lie(G)$ consisting of nilpotent matrices, and it is defined over $\Q$. So its associated Lie group, say $H$, is an abelian unipotent subgroup of $G$ defined over $\Q$. Since $x\in\SL(n+1,\Q)/\Lambda$, $H$ intersects the stabilizer of $x$ in a co-compact lattice. Hence $H\phi(s)x$ is compact. Therefore by Remark~\ref{rem:Hs-Bs}, we have that $H(s)=H$. Since $H\not\supset G$, $s\in E_x$.
\end{proof}

\subsection*{Acknowledgements} We would like to thank Manfred Einsiedler for insightful discussions on this topic. We are thankful to Dmitry Kleinbock for drawing our attention to the article of Pyartli. We thank the referees of this article for very careful reading and valuable suggestions for improvements. We again want to thank Dmitry Kleinbock for his input, which significantly improved the clarity and readability of this article.

\end{document}